\documentclass[11pt]{article}
 \usepackage[a4paper]{anysize}\marginsize{2cm}{2cm}{2cm}{2cm}
 \pdfpagewidth=\paperwidth \pdfpageheight=\paperheight
 \usepackage{amsfonts,amssymb,amsthm,amsmath,eucal,tabu,url}
 \usepackage{pgf}
 \usepackage{array}
 \usepackage{pstricks}
 \usepackage{pstricks-add}
 \usepackage{pgf,tikz}
 \usetikzlibrary{automata}
 \usetikzlibrary{arrows}
 \usepackage{indentfirst}
\theoremstyle{plain}
\newtheorem{thm}{Theorem}[section]
\newtheorem{theorem}[thm]{Theorem}

\newtheorem{proposition}[thm]{Proposition}
\newtheorem{corollary}[thm]{Corollary}
\newtheorem{conjecture}[thm]{Conjecture}

\theoremstyle{definition}
\newtheorem{definition}[thm]{Definition}
\newtheorem{remark}[thm]{Remark}
\newtheorem{example}[thm]{Example}

\newtheorem{problem}[thm]{Problem}

\newtheorem{thevarthm}[thm]{\varthmname}

\newenvironment{varthm*}[1]{\trivlist\item[]{\bf #1.}\it}{\endtrivlist}

\def\keywordname{{\bfseries Keywords}}%
\def\keywords#1{\par\addvspace\medskipamount{\rightskip=0pt plus1cm
\def\and{\ifhmode\unskip\nobreak\fi\ $\cdot$
}\noindent\keywordname\enspace\ignorespaces#1\par}}
\def\subclassname{{\bfseries Mathematics Subject Classification
(2010)}\enspace}
\def\subclass#1{\par\addvspace\medskipamount{\rightskip=0pt plus1cm
\def\and{\ifhmode\unskip\nobreak\fi\ $\cdot$
}\noindent\subclassname\ignorespaces#1\par}}

\begin{document}
\title{Hirzebruch-type inequalities viewed as tools in combinatorics}
\author{Piotr Pokora}
\date{\today}
\maketitle
\begin{abstract}
The main purpose of this survey is to provide an introduction, algebro-topological in nature, to Hirzebuch-type inequalities for plane curve arrangements in the complex projective plane. These inequalities gain more and more interest due to their utility in many combinatorial problems related to point or line arrangements in the plane. We would like to present a summary of the technicalities and also some recent applications, for instance in the context of the Weak Dirac Conjecture. We also advertise some open problems and questions.
\keywords{Hirzebruch-Kummer covers, curve arrangements, line arrangements, pseudoline arrangements, simplicial arrangements of lines, Dirac's conjecture, Beck's Theorem of two extremes}
\subclass{14N10, 52C35, 32S22, 14N20}
\end{abstract}
\section{Introduction}
In combinatorics, there are many interesting point-line incidence problems. Probably the most classical one is due to Sylvester \cite{Sylvester}.
\begin{problem}
\textit{Prove that it is not possible to arrange any finite number of real points so that a right line through every two of them shall pass through a third, unless they all lie in the same right line.}
\end{problem}
This problem is also related to the famous orchard problem proposed by Jackson as a \emph{rational amusement for winter evenings} \cite{Jackson}. Gallai \cite{Gallai} proved that Sylvester's problem has a positive answer.
\newpage
\begin{theorem}[Sylvester-Gallai]
Let $\mathcal{P}\subset \mathbb{R}^{2}$ be a finite set of points. Then either
\begin{itemize}
\item all points in $\mathcal{P}$ are collinear, or
\item there exists a line $\ell$ passing through exactly two points from $\mathcal{P}$.
\end{itemize}
\end{theorem}
There are several elegant proofs of this theorem. Probably the most instructive one is given by L. M. Kelly which can be found, for instance, in \cite{Proof}. Using duality in the projective plane we can formulate Sylvester-Gallai Theorem in the language of line arrangements and their intersection points, i.e., every line arrangement in the real projective plane consisting of at least $3$ lines, which is not a pencil, contains at least one double intersection point. This can be also observed using the well-known Melchior's inequality \cite{Melchior}. For an arrangement of lines $\mathcal{L} = \{\ell_{1}, ...,\ell_{d}\}$ in the projective plane we denote by $t_{r} = t_{r}(\mathcal{L})$ the number of $r$-fold points, i.e., points where exactly $r$-lines from the arrangement meet.
\begin{theorem}[Melchior]
Let $\mathcal{L} = \{\ell_{1}, ..., \ell_{d}\} \subset \mathbb{P}^{2}_{\mathbb{R}}$ be an arrangement of $d\geq 3$ lines. Assume that $\mathcal{L}$ is not a pencil, then
$$t_{2} \geq 3 + \sum_{r\geq 3}(r-3)t_{r}.$$
\end{theorem} 
Melchior's proof is based on a simple observation that every line arrangement in the real projective plane provides a partition of the space into $f$ regions, $e$ edges, and $v$ vertices, and then we can use the identity $v - e + f = e( \mathbb{P}^{2}_{\mathbb{R}}) = 1$. In fact, using the same method one can construct a whole series of Melchior-type inequalities, which seems to be a folklore result (this was shown for instance in a student paper \cite{Swornog}).
\begin{theorem}[Melchior-type inequality]
Let $\mathcal{L} = \{\ell_{1}, ..., \ell_{d}\} \subset \mathbb{P}^{2}_{\mathbb{R}}$ be an arrangement of $d\geq 3$ lines. Assume that $\mathcal{L}$ is not a pencil and pick $k \in \mathbb{Z}_{\geq 1}$, then
$$\sum_{r=2}^{2k}(2k+1-r)t_{r} \geq 2k+1 + \sum_{r \geq 2k+1} (r - (2k+1))t_{r}.$$
In particular, for $k=1$ we recover Melchior's inequality.
\end{theorem}
It was natural to ask whether Melchior's inequality can hold if we change the underlying  field, for instance if we consider a finite projective plane or the complex projective plane. In both cases the answer is negative. 

In the first case, consider $\mathbb{P}^{2}_{\mathbb{Z}_{2}}$ - the Fano plane. It is known that there exists a unique configuration of $7$ lines and $7$ points of multiplicity $3$, which obviously violates Melchior's inequality. Secondly, let us consider the following line arrangement in the complex projective plane defined by the linear factors of the polynomial
$$Q(x,y,z) = (z^{3}-y^{3})(y^{3}-z^{3})(x^{3}-z^{3}).$$
It can be seen that $Q$ defines the arrangement consisting of $9$ lines and $12$ triple intersection points, so it obviously violates Melchior's inequality. The arrangement defined by $Q$ is known as the dual Hesse arrangement of lines (or Ceva's arrangement of $9$ lines, as defined in \cite{IgorD}).

The above (counter)examples motivated researchers to find reasonable generalizations of Melchior's inequality (mostly over the complex numbers) involving the number of lines and $t_{r}$'s. It is worth mentioning that Iitaka \cite{Iitaka} claimed to prove (erroneously) that Melchior's inequality holds for line arrangements in the complex projective plane, which shows that the problem attracted the attention of people working in algebraic geometry. The breakthrough came with Hirzebruch's famous paper \cite{Hirzebruch}.
\begin{theorem}[Hirzebruch's inequality]
\label{17}
Let $\mathcal{L} = \{\ell_{1}, ..., \ell_{d}\} \subset \mathbb{P}^{2}_{\mathbb{C}}$ be an arrangement of $d\geq 4$ lines such that $t_{d} = t_{d-1} = 0$, then
$$t_{2} + t_{3} \geq d + \sum_{r\geq 5}(r-4)t_{r}.$$
\end{theorem}
It might be surprising that Hirzebruch's inequality is only a by-product of his construction, the Hirzebruch-Kummer cover of the complex projective plane branched along an arrangement of lines, which allowed him to construct new examples of algebraic surfaces of general type, so-called ball-quotients. We are not going into technicalities related to ball-quotient surfaces, but for interested readers we refer to the following classical textbook \cite{BHH87}. On the other hand, it turned out that Hirzebruch's inequality is an extremely important tool in numerous problems in combinatorial geometry, for instance, as it was advertised in \cite{Solymosi}, Hirzebruch's inequality can be applied in the context of Sylvester-Gallai type theorems over the complex numbers.

Our scope in this survey is to present an accessible outline of Hirzebruch's paper and other strong Hirzebruch-type inequalities which allowed researches to make progress on classical conjectures in combinatorics, like the Weak Dirac Conjecture \cite[Section 6]{Klee}. We hope that the survey will be useful for these combinatorialists who want to use Hirzebruch’s ideas. 

Our prerequisites are not demanding, basics on differential geometry and first lectures on algebraic geometry.

We work over the complex numbers, and we will use the natural inclusion of $\mathbb{R} \subset \mathbb{C}$.

\section{On Hirzebruch's inequality for line arrangements}
\subsection{Basics on algebraic surfaces}
Before we present a sketch of the proof of Hirzebruch's inequality, we recall some basics on algebraic surfaces. By an algebraic surface we mean an irreducible and reduced $2$-dimensional complex projective variety -- such a surface can be embedded into $\mathbb{P}^{N}_{\mathbb{C}}$ for some $N \in \mathbb{Z}_{>0}$. In algebraic geometry one of the most important problems is to classify objects considering a certain fixed set of invariants. In order to provide some intuition, we begin with one-dimensional complex varieties, i.e., smooth algebraic curves. The basic invariant of an algebraic curve is its \emph{genus} $g$. This can be explained topologically when considering a complex algebraic curve as an oriented real surface. For example, a two-dimensional sphere is an algebraic curve of genus $0$. A torus (a donut-shaped surface) is an algebraic curve of genus $1$. A very important concept in algebraic geometry is that of divisors. A divisor on an algebraic curve $C$ is a finite sum of points in $C$ with integer coefficients, i.e., an expression of the form
$$D = \sum_{i=1}^{r} a_{i}P_{i},$$
where $a_{i} \in \mathbb{Z}$ and $P_{i} \in C$ for $i \in \{1, ...,r\}$. To any function $f$ (or more generally, a section of a line bundle) on $C$, one can associate its divisor, which intuitively can be thought of as \emph{zeroes} of $f$ minus \emph{poles} of $f$. Whereas the existence of line bundles on an algebraic variety $X$ is a subtle problem, there is always its tangent bundle $T_{X}$. If $C$ is a curve, then $T_{C}$ is a line bundle and the degree of a section in $T_{C}$ (i.e. of a vector field on $C$) is $2-2g(C)$, which leads to an algebraic point of view on the notion of the genus. For historical reasons, it is customary to study sections of the dual bundle $T_{C}^{*}$ which is the canonical bundle $K_{C}$ of $C$. Thus we have ${\rm deg} \, K_{C} = 2g(C) - 2$. Similarly, for an algebraic surface $X$ one can consider its tangent bundle $T_{X}$ (which is a vector bundle of rank $2$) and the canonical line bundle, which is the determinantal line bundle of its dual $K_{X} = {\rm det} \, T^{*}_{X}$. The canonical divisor, which by a slight abuse of notation we also denote by $K_{X}$, is the difference between zeros and poles of a section in the line bundle $K_{X}$. Formally, it is a sum of codimension $1$ subvarieties in $X$ (curves) with integer coefficients. Note that this is in full analogy to divisors on curves, where we have sums of codimension $1$ subvarieties -- points in this case. Curves on surfaces intersect in points, thus one can make sense of the self-intersection of $K_{X}$. This integer is denoted by $K_{X}^{2}$ and plays the role of the number ${\rm deg} \, K_{C}$ for curves. However, for surfaces $K_{X}^{2}$ alone does not provide enough information and one introduces another invariant, the topological Euler-Poincar\'e characteristic $e(X)$. In the case of curves, $e(C)$ is equal to ${\rm deg} \, K_{C}$. In the case of surfaces, these two natural generalizations of the degree of the canonical divisor differ and one studies the pair $(K_{X}^{2},e(X))$. The problem of establishing which pairs of integers $(m,n)$ may appear as $m=K_{X}^{2}$ and $n=e(X)$ is known as the geography problem, and it is not completely solved yet.
If the line bundle $K_{X}$ (or its tensor powers $K_{X}^{\otimes m}$, which we customarily write in the additive notation as $mK_{X}$) has many global sections, i.e., when $$ \dim H^{0}(X, mK_{X}) \sim c \cdot m^{2}$$
for a positive constant $c$, then $X$ is said to be of \emph{general type}. 
The most important constraint on the existence of algebraic surfaces of general type is the celebrated Bogomolov-Miyaoka-Yau inequality (see for instance \cite{M84,Wahl})
\begin{equation}
\label{BMY}
K_{X}^{2} \leq 3e(X).
\end{equation}
In fact, this inequality remains true the under milder assumption that the canonical divisor has asymptotically some sections, i.e., 
\begin{equation}
\label{eq:Kod}
{\rm dim} \, H^{0}(K, mK_{X}) > 0
\end{equation}
for some $m$ sufficiently large. We say that $X$ satisfying \eqref{eq:Kod} has non-negative Kodaira dimension. It is natural to wonder when in \eqref{BMY} there is equality. It turns out that there is an elegant topological answer to this question. Recall that a covering $p: Y \rightarrow X$ of a topological space $X$ is \emph{universal} if $Y$ is simply-connected. We have the following fundamental result, see for instance \cite{Miyaoka77}.
\begin{theorem}
Let $X$ be a smooth complex projective surface of general type. Then \eqref{BMY} holds with equality if and only if the universal cover of $X$ is the complex unit ball $\{(z_{1},z_{2}) \in \mathbb{C}^{2} : |z_{1}|^{2} + |z_{2}|^{2} < 1\}$, i.e., $X$ is a ball-quotient. 
\end{theorem}
The author suggests to consult \cite{Barth}, as a one of various possible sources, for a comprehensive introduction to the theory of algebraic surfaces.
\subsection{Hirzebruch's construction}
Covers of algebraic surfaces play an important role in Hirzebruch's approach. Here we briefly outline some basic properties of such covers. For more details we refer to \cite{Hironaka}.

\begin{definition}
A branched covering $\rho : X \rightarrow Y$ is a finite surjective morphism between normal varieties. Denote by $G$ the group of automorphisms $\alpha : X \rightarrow X$ so that $\rho(\alpha(x)) = \rho(x)$ for all $x \in X$. The group $G$ is called the group of covering automorphisms of $\rho$. If $G$ acts transitively on all fibers of our cover $\rho$, then the covering is called \emph{Galois}. We say that a branched covering $\rho : X \rightarrow Y$ is an abelian covering if $\rho: X \rightarrow Y$ is Galois and additionally the group of covering automorphisms is abelian.
\end{definition}
We also need to introduce the following notation (from Hirzebruch's papers), namely if $\mathcal{L} \subset \mathbb{P}^{2}_{\mathbb{C}}$ is an arrangement of lines, then
$$f_{0}=\sum_{r\geq 2}t_{r}, \quad f_{1} = \sum_{r\geq 2}rt_{r}.$$
Now we are ready to present the main result of this section. We will provide a detailed outline of the proof emphasizing a topological part of Hirzebruch's considerations. Our outline is still quite technical, and might be challenging, but we are doing this to emphasize the places where algebraic geometry methods are decisive and might be difficult to replace by combinatorial techniques.
\begin{theorem}
\label{thm:hir}
Let $\mathcal{L} = \{\ell_{1}, ... ,\ell_{d}\} \subset \mathbb{P}^{2}_{\mathbb{C}}$ be an arrangement of $d\geq 6$ lines such that $t_{d} = t_{d-1} = 0$. Then
\begin{equation}
\label{Hir}
t_{2} + t_{3} \geq d + \sum_{r \geq 5}(r-4)t_{r}.
\end{equation}
\end{theorem}
\begin{proof}
Here is the strategy. The key idea of Hirzebruch is to use abelian coverings of the complex projective plane branched along line arrangements. This idea leads to interesting algebraic surfaces for which the self-intersection of canonical divisor and the topological Euler-Poincar\'e characteristic can be expressed in terms of the combinatorics of a given arrangement. Under the conditions that we have at least $d\geq 6$ lines and $t_{d} = t_{d-1} = 0$, we can deduce that our newly constructed surface is of non-negative Kodaira dimension. Thus we can apply the Bogomolov-Miyaoka-Yau inequality \eqref{BMY}.

Starting from scratch, and following ideas from \cite[Section 4]{Catanese}, let us denote by $s_{i}$ 
the defining linear forms of $\ell_{i}$, i.e., $\ell_{i} = V(s_{i})$, for all $i \in \{1, ..., d\}$. Now we consider the following map 
$$f : \mathbb{P}^{2}_{{\mathbb{C}}} \ni x \mapsto (s_{1}(x): ... : s_{d}(x)) \in \mathbb{P}^{d-1}_{\mathbb{C}}.$$
Let us emphasize that $f$ is well-defined since, by the assumption there is no point where all
lines meet, at every point at least one of the $s_{j}(x)$'s is non-zero.
Now we use the Kummer covering
$${\rm Km}_{n}: \mathbb{P}^{d-1}_{\mathbb{C}} \ni (y_{1}: ... :y_{d}) \mapsto (y_{1}^{n}: ... : y_{d}^{n}) \in \mathbb{P}^{d-1}_{\mathbb{C}},$$
where $n\geq 2$ is called the exponent of ${\rm Km}_{n}$. One can show that this covering is of degree $n^{d-1}$ with the Galois group $(\mathbb{Z}/n\mathbb{Z})^{d-1}$. Obviously, it is branched along $y_{1} \cdot ... \cdot y_{d} = 0$. Our main object of interest is the following fiber product:
\begin{equation}
\label{XN}
X_{n} := \mathbb{P}^{2}_{\mathbb{C}} \times_{\mathbb{P}^{d-1}_{\mathbb{C}}} \mathbb{P}^{d-1}_{\mathbb{C}} = \{(x,y) \in \mathbb{P}^{2}_{\mathbb{C}} \times \mathbb{P}^{d-1}_{\mathbb{C}} : f(x) = {\rm Km}_{n}(y)\}.
\end{equation}
There exists a projective transformation on $\mathbb{P}^{2}_{\mathbb{C}}$ such that $\ell_{1} = \{x_{1}=0\}$, $\ell_{2} = \{x_{2} = 0\}$, and $\ell_{3} = \{x_{3} = 0 \}$. Then we can describe $X_{n}$ even more explicitly looking at this surface as embedded in $\mathbb{P}^{d-1}$ (the second factor in the fibre product above). Indeed, $X_{n}$ is given by equations
$$
X_{n} = \{ (y_{1}: ... :y_{d})\in \mathbb{P}^{d-1}_{\mathbb{C}} : y_{j}^{n} = s_{j}(y_{1}^{n}, y_{2}^{n}, y_{3}^{n}) \,\, {\rm for} \,\, j \in \{4, ... ,d\} \}.
$$
In this explicit description, our surface $X_{n}$ is given by $(d-3)$-homogeneous equations in $\mathbb{P}^{d-1}_{\mathbb{C}}$, which means that $X_{n}$ is a {\it complete intersection}. One can show (using a local argument) that $X_{n}$ is singular over a point $p$ of the arrangement $\mathcal{L}$ if and only if $p$ is a point of multiplicity $\geq 3$, so that, unless $\mathcal{L}$ has only double points, $X_{n}$ is singular. We need to pass to its \emph{desingularization} (a smooth surface $Y_{n}$ and a generically one-to-one morphism $\tau: Y_{n} \rightarrow X_{n}$) in order to apply inequality \eqref{BMY}. Since $Y_{n}$ is a smooth complex projective surface, we can compute the following numbers (all relevant computations can be found in \cite[pp. 123-125]{Hirzebruch}): 
$$K^{2}_{Y_{n}}/n^{d-3}= n^{2}(9-5d + 3f_{1} - 4f_{0}) + 4n(d-f_{1}+f_{0})+f_{1}-f_{0}+d+t_{2},$$
$$e(Y_{n})/n^{d-3} = n^{2}(3-2d+f_{1}-f_{0}) + 2n(d-f_{1}+f_{0}) + f_{1}-t_{2}.$$
In the next step, quite cumbersome, one needs to check under which conditions on the incidence distribution of $\mathcal{L}$ our surface has non-negative Kodaira dimension -- it turns out that it is enough to assume that $d\geq 6$, $t_{d}=t_{d-1} = 0$, and $n \geq 3$. Then the Bogomolov-Miyaoka-Yau inequality $K^{2}_{Y_{n}} \leq 3e(Y_{n})$ implies that the following Hirzebruch polynomial
$$H_{\mathcal{L}}(n) = \frac{3e(Y_{n}) - K^{2}_{Y_{n}}}{n^{d-3}} = n^{2}(f_{0}-d) + 2n(d-f_{1}+f_{0}) + 2f_{1}+f_{0}-d-4t_{2}$$
is non-negative for $n\geq 3$. Evaluating $H_{\mathcal{L}}$ at $n=3$ we get 
$$t_{2}+t_{3} \geq d +\sum_{r\geq 5}(r-4)t_{r},$$
which is exactly Hirzebruch's inequality.
\end{proof}
\begin{remark}
Hirzebruch's inequality implies that every configuration of $d\geq 6$ lines with $t_{d}=t_{d-1}=0$ contains double or triple points as the intersections. 
\end{remark}
\begin{remark}
It is natural to ask whether Hirzebruch's inequality is sharp, i.e., whether there exists a line arrangement $\mathcal{A}$ such that $t_{2}+t_{3} = d + \sum_{r\geq 5}(r-4)t_{r}$. There exists exactly one (!) arrangement of lines satisfying the above equality, namely the Hesse arrangement of lines. This arrangement consists of $d=12$ lines having $t_{2}=12$ and $t_{4}=9$. The proof of this quite surprising result is not elementary (in its full generality), one needs to use the theory of totally geodesic curves in complex compact ball-quotients \cite{BHH87}. In the case when we restrict our attention to real line arrangements, we refer to \cite{BokowskiP} for an elementary proof of the fact that there are no such arrangements.
\end{remark}
\begin{remark}
In the same paper \cite{Hirzebruch}, Hirzebruch defines the so-called characteristic numbers of line arrangements, namely
$$\gamma(\mathcal{L}) = \lim_{n \rightarrow \infty} \frac{K^2_{Y_{n}}}{e(Y_{n})} = \frac{9-5d + 3f_{1}-4f_{0}}{3-2d+f_{1}-f_{0}}.$$
Somesse \cite{Som} proved that for complex line arrangements, 
$$\gamma(\mathcal{L}) \leq \frac{8}{3},$$
with equality if and only if $\mathcal{L}$ is the dual-Hesse arrangement of lines. This result, in particular, implies that if $\mathcal{L}$ is an arrangement of $d\geq 6$ lines with $t_{d}=t_{d-1}=0$, then
$$2t_{2} + t_{3} \geq 3 + d + \sum_{r\geq 5}(r-4)t_{r}.$$
Observe that $\gamma(\mathcal{L}) = 3$ implies that $f_{0}=d$, and by the Erd\H{o}s-de Bruijn Theorem \cite{deBr} this condition forces $\mathcal{L}$ to be a near-pencil, i.e., an arrangement of $d$ lines such that $t_{d-1}=1$ and $t_{2}=d-1$. Note that for a near-pencil, the above inequality also holds:
$$2d-2 = 2t_{2} + t_{3} \geq 3 + d + \sum_{r\geq 5}(r-4)t_{r} = 2d-2.$$
\end{remark}
\begin{remark}
Hirzebruch's construction provides a whole series of inequalities depending on $n\geq 3$. In particular, evaluating $H_{\mathcal{L}}$ at $n=5$ we obtain
$$4t_{2}+3t_{3}+t_{4} \geq 4d + \sum_{r\geq 5}(2r-9)t_{r}.$$
It is natural to ask whether this inequality is sharp, and it turns out that there exists exactly one real line arrangement providing equality, the well-known $\mathcal{A}_{1}(6)$ configuration consisting of $d=6$ lines and $t_{3}=4$, $t_{2}=3$. For a combinatorial proof of this statement we refer to \cite{BokowskiP}. Moreover, one can show that there is exactly one line arrangement defined over the complex numbers providing equality, the dual-Hesse arrangement of $9$ lines and $12$ triple points.
\end{remark}
\begin{remark}
Using finer considerations on the Kodaira dimension of $Y_{n}$, we can show that if $d\geq 6$ with $t_{d}=t_{d-1}=t_{d-2}=0$ and $n\geq 2$, then our surface $Y_{n}$ has non-negative Kodaira dimension \cite[Kapitel 3] {BHH87}. The condition $H_{\mathcal{L}}(2)\geq 0$ leads us to
$$t_{2}+3t_{3}+t_{4} \geq d + \sum_{r\geq 5}(2r-9)t_{r}.$$
\end{remark}
\begin{remark}
In the literature, we can find usually the following variant of Hirzebruch's inequality
\begin{equation}
\label{Hirz2}
t_{2} + \frac{3}{4} t_{3} \geq d + \sum_{r \geq 5}(2r-9)t_{r}
\end{equation}
provided that $d\geq 6$ and $t_{d}=t_{d-1}=t_{d-2}=0$. In order to justify this claim, one needs to use Miyaoka-Sakai's improvement \cite{Hirzebruch1,M84,Sakai} of the Bogomolov-Miyaoka-Yau inequality which says that if $Y_{n}$ contains either smooth rational curves (genus = $0$) or smooth elliptic curves (genus = $1$), then one always has $3c_{2}(Y_{n}) - c_{1}^{2}(Y_{n}) \geq  const > 0$, and the number $const$ can be explicitly determined -- in fact it is given in geometric terms. This leads to the desired inequality.
\end{remark}
\begin{remark}
In \emph{Extremal problems in combinatorial geometry} by Erd\H{o}s and Purdy \cite{EP}, Section 4.1.1 is devoted to Hirzebruch's inequalities. Erd\H{o}s asks here whether one can provide elementary and independent proofs of Hirzebruch's inequalities provided that we restrict our attention to $\mathbb{P}^{2}_{\mathbb{R}}$.
Moreover, in \textit{Research Problems in Discrete Geometry} by Brass, Moser, and Pach \cite[p.~315; Problem 7]{Research} one of the stated research problems is to prove Hirzebruch's inequality (\ref{Hirz2}) using only elementary methods. In the light of what we have seen so far, this seems to be extremely difficult. The main ingredient of Hirzebruch's construction is the Bogomolov-Miyaoka-Yau inequality which is not combinatorial in its nature. The next section presents even stronger inequalities involving the
number of lines and intersection points that also follow from variants of the Bogomolov-Miyaoka-Yau inequality. At this stage, at least to the author, it seems that there is no hope to find an easy proof of (\ref{Hirz2}).
\end{remark}

\begin{remark}
It is easy to observe that every configuration of $d \in \{4,5\}$ lines with $t_{d}=t_{d-1}=0$ also satisfies Hirzebruch's inequality (\ref{Hir}), therefore the formulation of Theorem \ref{17} is in fact equivalent to that of Theorem \ref{thm:hir}.
\end{remark}

Before we pass to (stronger) Hirzebruch-type inequalities, let us present an interesting way to construct K3 surfaces using abelian covers branched along $6$ general lines.
\begin{definition}
A smooth complex surface $X$ is called a K3 surface if it is a simply-connected compact complex manifold of dimension $2$ such that the canonical divisor is trivial. In particular, its Kodaira dimension is zero and $K_{X}^{2}=0$.
\end{definition}
\begin{example}
Consider $\mathcal{L} = \{\ell_{1}, ..., \ell_{6}\} \subset \mathbb{P}^{2}_{\mathbb{C}}$ an arrangement of $6$ generic lines which means that the only intersection points of these lines are double points. We can find a projective transformation such that $\ell_{1} = \{x=0\}, \ell_{2}=\{y=0\}$, and $\ell_{3}=\{z=0\}$. We denote by $\ell_{i}=a_{i}x + b_{i}y+c_{i}z$ with $i \in \{4,5,6\}$ the equations of remaining $3$ lines. Now we can consider the Hirzebruch-Kummer cover $X_{2}$ with exponent $n=2$ branched along $\ell_{1}, ..., \ell_{6}$. We know that $X_{2}$ is a smooth projective surface and it can be described as
$$X_{2} = \{ (z_{1},z_{2},z_{3},z_{4},z_{5},z_{6}) \in \mathbb{P}^{5}_{\mathbb{C}} : z^{2}_{i} = a_{i}z_{1}^{2}+b_{i}z_{2}^{2}+c_{i}z_{3}^{2}, i \in \{4,5,6\} \},$$
so our surface $X_{2}$ is a smooth complete intersection of $3$ quadrics in $\mathbb{P}^{5}_{\mathbb{C}}$. This surface is well-known in algebraic geometry, i.e., $X_{2}$ is a K3 surface of degree $8$. Let us conclude this remark by the following algebraic connection between $X_{2}$ and two-to-one covering of the complex projective plane branched along $\ell_{1}, ...,\ell_{6}$ -- it turns out that $X_{2}$ is the minimal desingularization of this covering, please consult \cite[p. 770]{GH} for details.
\end{example}

\section{Stronger Hirzebruch-type inequalities for complex line arrangements}
Now we present (stronger) Hirzebruch-type inequalities for line arrangements in the complex projective plane. These results follow from Langer's version of the orbifold Miyaoka-Yau inequality for normal surfaces with boundary divisors. Since Langer's result is highly non-trivial (it involves, for instance, the notion of orbifold Euler numbers, and other technical considerations), we do not provide details -- motivated readers can consult \cite{Langer}.

Let us start with the first strong Hirzebruch's type inequality, which was first proved by Bojanowski \cite{Bojanowski} in his Master Thesis (in Polish).  
\begin{theorem}
\label{Bojan}
Let $\mathcal{L}=\{\ell_{1}, ...,\ell_{d}\} \subset \mathbb{P}^{2}_{\mathbb{C}}$ be a line arrangement with $d\geq 6$ such that $t_{r} = 0$ for $r >\frac{2d}{3}$. Then 
\begin{equation}
\label{langorb}
t_{2} + \frac{3}{4}t_{3} \geq d + \sum_{r\geq 5}\bigg( \frac{r^{2}}{4}-r\bigg)t_{r}.
\end{equation}
\end{theorem}
One proof of this result can be deduced from \cite[Theorem 2.2]{Pokora2} with $d=1$. It also follows from the following two inequalities for complex line arrangements due to Langer \cite[Proposition 11.3.1]{Langer}.

\begin{theorem}
\label{lang}
Let $\mathcal{L} = \{\ell_{1}, ..., \ell_{d}\}\subset\mathbb{P}^{2}_{\mathbb{C}}$ be a line arrangement such that $t_{r}=0$ for $r >\frac{2d}{3}$. Then 
$$\sum_{r\geq 2}r^{2}t_{r}  \geq \bigg\lceil \frac{4d^{2}}{3} \bigg\rceil,$$
$$\sum_{r\geq 2}rt_{r} \geq \bigg\lceil \frac{d^{2}}{3} + d \bigg\rceil.$$
\end{theorem}
It is natural to compare Bojanowski's version of Hirzebruch's inequality with others, and we can easily observe the following chain of inequalities (under the assumption that $t_{r}=0$ for $r> \frac{2d}{3}$):

$$t_{2} + t_{3} \geq t_{2} + \frac{3}{4}t_{3} \geq d + \sum_{r\geq 5} \bigg( \frac{r^{2}}{4}-r\bigg)t_{r} \geq d + \sum_{r\geq 5}(2r-9)t_{r} \geq d + \sum_{r\geq 5}(r-4)t_{r}.$$

Let us now list examples of line arrangements\footnote{If you would like to learn more about these arrangements and the geometry lurking behind them, we refer to \cite{BHH87,IgorD} for details.} for which we obtain equality in (\ref{langorb}) -- our list is probably far from complete.

\begin{enumerate}
\item \emph{Icosahedron arrangement} consisting of $15$ lines and $t_{2} = 15, t_{3} = 10, t_{5}=6$.
\item \emph{Ceva's arrangements} consisting of $3n$ lines ($n\geq 4$), and $t_{3} = n^{2}, t_{n}=3$.
\item \emph{The extended Ceva's arrangements} consisting of $3n+3$ lines with $n\geq 3$, and $t_{2} = 3n$, $t_{3}=n^{2}$, $t_{n+2}=3$.
\item \emph{The Hesse arrangement} consisting of $12$ lines and $t_{4}=9, t_{2} = 12$.
\item The union of Ceva's arrangement of $9$ lines and the Hesse arrangement consisting of $d=12+9$ lines with $t_{2}=36$, $t_{4}=9$, $t_{5}=12$.
\item \emph{Klein's arrangement} consisting of $21$ lines and $t_{3}=28, t_{4}=21$.
\item \emph{Wiman's arrangement} consisting of $45$ lines and $t_{3}=120, t_{4}=45, t_{5}=36$.
\end{enumerate}

There exists an infinite series of line arrangements such that equality in (\ref{langorb}) holds -- for instance Ceva's line arrangements. Moreover, note that there exists an interesting incidence distribution $\mathcal{C}$ constructed in \cite[p.~116]{BHH87}. It consists of $d=12m+3$ lines and $t_{2} = 12m^{2}+15m+3$, $t_{6}=4m^{2} + m$ with $m \in \mathbb{Z}_{\geq 3}$. It can be shown that this incidence distribution cannot be realized over the real numbers (i.e., there does not exist any line arrangement defined over the real numbers possessing the mentioned distribution). This leads to the first open problem of this survey.
\begin{problem}
Is it possible to construct arrangements of $d=12m+3$ lines in the complex projective plane such that  $t_{2}=12m^{2}+15m+3$, $t_{6}=4m^{2}+m$ with $m\in \mathbb{Z}_{\geq 3}$?
\end{problem}
Simple calculations reveal that the distribution $\mathcal{C}$ satisfies $(\ref{langorb})$ with equality and if one can show that there exists $m_{0}\in \mathbb{Z}_{\geq 3}$ for which we can realize $\mathcal{C}$ over the complex numbers, then $\mathcal{C}$ leads to a new example of complex and compact $2$-dimensional ball-quotient (in fact this is the main reason why this problem is really attractive).

Now we are in a good position to present (probably) the strongest known Hirzebruch-type inequality for complex line arrangements. The inequality in question is the main result of Bojanowski's thesis \cite[Theorem 2.3]{Bojanowski}.
\begin{theorem}
Let $\mathcal{L} \subset \mathbb{P}^{2}_{\mathbb{C}}$ be an arrangement of $d$ lines. Pick a natural number $n \in [3, ..., d)$ and assume that $t_{r} = 0$ for $r \geq d-n+2$. Then
$$t_{2}+\frac{3}{4}t_{3} \geq d + \sum_{r=5}^{s-1}\bigg(\frac{r^{2}}{4}-r\bigg)t_{r} + \sum_{r=s}^{d-n}((n-1)r - n^{2})t_{r} + \bigg((n-2)(d-n+1) - (n-1)^{2}\bigg)t_{d-n+1},$$
where $s = \min \{2n, d-n\}$.
\end{theorem}
\section{Applications}
In this section, we focus on applications of Hirzebruch-type inequalities in the context of interesting combinatorial problems in incidence point-line theory. We present only three aspects in order to avoid repetitions. For more applications of Langer's inequalities, and in some sense Hirzebruch-type inequalities, we refer for instance to a recent paper by Frank de Zeeuw  \cite{Frank}.
\subsection{The Weak Dirac Conjecture}
Let us denote by $\mathcal{P}\subset \mathbb{P}^{2}_{\mathbb{C}}$ a finite set of $n$ mutually distinct points and let $\mathcal{L}(\mathcal{P})$ be the set of lines determined by $\mathcal{P}$, where a line that passes through at least two points from $\mathcal{P}$ is said to be determined by $\mathcal{P}$. We denote by $l_{r}$ the number of $r$-rich lines determined by exactly $r$ points from $\mathcal{P}$.

As a starting point for our discussion we recall the original Dirac conjecture \cite{Dirac}. Note that Dirac never conjectured this in print, although he states twice in \cite{Dirac} that its truth is \emph{likely}.
\begin{conjecture}[Dirac]
Every set $\mathcal{P}$ of $n$ non-collinear points contains a point in at least $\frac{n}{2}$ lines determined by $\mathcal{P}$.
\end{conjecture}
It turned out that the Dirac conjecture is false -- the smallest counterexample has $n=7$ points, namely the vertices of a triangle together with the midpoints of its sides and its centroid. However, the conjecture was resolved positively by Green and Tao in \cite{GT} for very large $n$. In this view, we can formulate the actual Dirac conjecture which is, according to our best knowledge, open.
\begin{conjecture}
There is a constant $c$ such that every set $\mathcal{P}$ of $n$ non-collinear points contains a point in at least $\frac{n}{2}-c$ lines determined by $\mathcal{P}$.
\end{conjecture}

In 1961, P. Erd\H{o}s proposed the following Weak Dirac Conjecture \cite{probE}.
\begin{conjecture}[WDC]
Every set $\mathcal{P}$ of $n$ non-collinear points in the plane (presumably over the real numbers) contains a point which is incident to at least $\lceil \frac{n}{c} \rceil$ lines from $\mathcal{L}(\mathcal{P})$ for some constant $c > 0$. 
\end{conjecture}
The Weak Dirac Conjecture was proved independently by Beck \cite{Beck} and Szemer\'edi-Trotter \cite{Semeredi}, but they did not specify the actual value of $c$. In 2012, Payne and Wood showed the WDC with $c=37$ \cite{Wood}, and one of the main ingredients of their proof is Hirzebruch's inequality. 

On the other hand, as we can read in \cite[Chapter 6]{Klee}, it was more plausible to believe that $c = 3$, and it turned out that this prediction is correct \cite{Zeye}.
\begin{theorem}[Han]
The Weak Dirac Conjecture holds with $c=3$.
\end{theorem}
\begin{proof}
We will follow Han's approach \cite{Zeye}. First of all, note that if $\mathcal{P}$ is a finite set of non-collinear points and it contains at least $\lceil\frac{n}{3}\rceil+1$ points which lie on a line $\ell$, then we are done -- it is enough to consider a point $p \in \mathcal{P} \setminus \ell$ which is incident, by definition/construction, to at least $\lceil\frac{n}{3}\rceil+1$ lines, so we may assume that $\mathcal{P}$ does not contain $\lceil\frac{n}{3}\rceil+1$ collinear points. According to the dual version of Bojanowski's inequality, we have $$l_2+\frac{3}{4}l_3\geq n+\sum_{r\geq 5}\bigg(\frac{r^2}{4}-r\bigg)l_r ,$$ which can be written as $$l_2+\frac{3}{4}l_3\geq n+\sum_{r\geq 5}\frac{\binom{r}{2}}{2}l_r-\frac{3}{4}\sum_{r\geq 5}rl_r.$$
Using the combinatorial count $\binom{n}{2} = \sum_{r\geq 2} \binom{r}{2}l_{r}$, we obtain
$$l_2+\frac{3}{4}l_3\geq n+\frac{{n\choose2}}{2}-\sum_{r= 2}^4\frac{{r\choose2}}{2}l_r-\frac{3}{4}\sum_{r\geq 5}rl_r$$
This gives $$\sum_{r\geq2}rl_r\geq\frac{n(n+3)}{3},$$
and we finally obtain $$\sum_{p\in\mathcal{P}}{\rm mult}_{p} \geq\frac{n(n+3)}{3}.$$
Using the Pigeonhole Principle, there exists a point from $\mathcal{P}$ which is incident to at least $\lceil\frac{n}{3}\rceil+1$ lines from $\mathcal{L}(\mathcal{P})$, and it completes the proof.
\end{proof}
It is worth mentioning that Han's result can be also obtained using Langer's inequality (cf. \cite[Corollary 1.2]{Frank}). 
\subsection{Beck's theorem on two extremes}
In this subsection, we would like to report on some progress towards better estimations in Beck's Theorem \cite[Theorem 3.1]{Beck}.
\begin{theorem}[Beck]
For a finite set $\mathcal{P}$ of $n$ points in $\mathbb{R}^{2}$ one of the following is true:
\begin{itemize}
\item there exists a line that contains $c_{1}  n$ points from $\mathcal{P}$ for some positive $c_{1}$;
\item there are at least $c_{2} n^2$ lines determined by $\mathcal{P}$.
\end{itemize}
\end{theorem}
Beck in his paper gave $c_{1} = \frac{1}{100}$ and $c_{2}$ was unspecified, Payne and Wood in \cite[Theorem 5]{Wood} provided $c_{1} = c_{2} = \frac{1}{100}$. Using Langer's inequality, Frank de Zeeuw observed \cite[Theorem 2.1]{Frank} that one can significantly improve estimations on $c_{1}$ and $c_{2}$.
\begin{theorem}[de Zeeuw]
\label{strongbeck}
Let $\mathcal{P}$ be a finite set of $n$ points in $\mathbb{R}^{2}$, then one of the following is true:
\begin{itemize}
\item there is a line that contains more than $\frac{6 + \sqrt{3}}{9}n$ points of $\mathcal{P}$;
\item there are at least $\frac{n^{2}}{9}$ lines determined by $\mathcal{P}$.
\end{itemize}
\end{theorem}

\begin{proof}
First suppose that $\mathcal{P}$ has at most $2n/3$ collinear points. By the dual version of Langer's inequality, 
$$\sum_{r\geq 2} r l_{r} \geq \frac{n(n+3)}{3}.$$
If we add Langer's inequality to the dual version of Melchior's inequality
$$\sum_{r\geq 2} (3-r) l_r\geq 3$$
we obtain
$$\sum_{r\geq 2} r l_r + \sum_{r\geq 2} (3-r) l_r
\geq \frac{n(n+3)}{3} + 3,$$
or equivalently
\begin{equation*}
3\cdot |\mathcal{L}(\mathcal{P})| = 3 \cdot \sum_{i\geq 2} l_i \geq \frac{n^2+3n+9}{3}.
\end{equation*}
Thus $|\mathcal{L}(\mathcal{P})| \geq n^2/9$, which proves the second alternative.

Assume that $\mathcal{P}$ has more than $2n/3$ collinear points. Let $\bar{\ell}$ be the line that contains more than $2n/3$ points from $\mathcal{P}$ such that $|\mathcal{P} \cap \bar{\ell}| = \alpha  n$ and $|\mathcal{P} \, \setminus \, \mathcal{P}'| = (1-\alpha) n$, where $\mathcal{P}' \subset \mathcal{P}$ is  the set of points that is contained in $\bar{\ell}$.
We now lower bound $|\mathcal{L}(P)|$. Count one line for every choice of a point from $\mathcal{P} \cap \bar{\ell}$ and a point from $\mathcal{P} \setminus \mathcal{P}'$, but we may overcount by one for every pair of points from $\mathcal{P} \setminus \mathcal{P}'$ when the line through that pair hits $\bar{\ell}$ in a point from $\mathcal{P}$. This leads to 
\[|\mathcal{L}(\mathcal{P})| = \sum_{r\geq 2} \ell_{r} \geq \alpha n\cdot (1-\alpha )n - \binom{(1-\alpha )n}{2} \geq \left(-\frac{3}{2}\alpha^2 +2\alpha  - \frac{1}{2}\right)n^2.
\]
As long as $\bigg(-3\alpha^2/2 +2\alpha - 1/2 \bigg) \geq 1/9$, the second alternative holds. Solving this quadratic inequality with respect to $\alpha$ we see that this is the case when $\alpha\leq (6+\sqrt{3})/9$.
Otherwise, the first alternative holds, and this completes the proof.
\end{proof}
In the light of Beck's theorem on two extremes, we can ask whether there exists a reasonable lower bound on the number of lines that are determined by a few points. The following result, which can be viewed as a corollary to Theorem \ref{strongbeck}, provides a surprising answer in the case of lines determined by two and three points.
\begin{corollary} 
Let $\mathcal{P}$ be a set of $n$ points in $\mathbb{R}^2$ such that at most $\alpha  n$ points are collinear with $\alpha = (6+\sqrt{3})/9$. Then
\[ l_2+ l_3 \geq \frac{n^2}{18}.\]
\end{corollary}
\begin{proof}
We add $l_2 + 2 l_3$ to the both sides of the dual version of Melchior's inequality obtaining
\[2 l_2+ 2 l_3 \geq 3 + l_2 + 2 l_3 + \sum_{i\geq 4} (i-3) l_i \geq 3 + \sum_{i\geq 2} l_i.\]
By Theorem \ref{strongbeck}, we have $$ l_2 + l_3 \geq \frac{1}{2}\cdot \sum_{r\geq 2} l_{r} \geq \frac{1}{2} \cdot \frac{n^{2}}{9} = \frac{n^2}{18}.$$
\end{proof}
It is natural to ask whether $n$ points in $\mathbb{C}^{2}$, with not too many collinear, determine a quadratic number of lines with at most three points. Following de Zeeuw \cite[Conjecture 4.5]{Frank1}, let us formulate the following conjecture.
\begin{conjecture}
There exists a constant $c > 0$ such that, if a set $\mathcal{P}$ of $n$ points in $\mathbb{C}^{2}$ has
at most $cn$ collinear, then $\mathcal{P}$ determines at least $cn^2$ lines with at most three points.
\end{conjecture}
According to my best knowledge, this conjecture is still open. However, we can show the following result involving $\ell_{2}, \ell_{3}, \ell_{4}$ -- it was presented by the author during the workshop \emph{Algebraic Geometry and Combinatorics} in January 2019 in  Loughborough.
\begin{proposition}
Let $\mathcal{P}$ be a subset of $n$ distinct points in $\mathbb{C}^{2}$ with at most $2n/3$ collinear.  Then
$$l_{2} + l_{3} + l_{4} \geq \frac{n(n+15)}{18} \approx \frac{n^{2}}{18}.$$
\end{proposition}
\begin{proof}
According to the dual version of Bojanowski's inequality, we have
$$l_{3} + \frac{3}{4} l_{3} \geq n + \sum_{r\geq 5}\bigg(\frac{r^{2}-4r}{4}\bigg) l_{r}.$$
Observe that for $r\geq 5$ one has  
$$\frac{r^{2}-4r}{4} \geq \frac{1}{8} \cdot \frac{r^{2}-r}{2},$$
and using the combinatorial count
$${ n \choose 2} = l_{2} + 3 l_{3} + 6 l_{4} + \sum_{r\geq 5}{r \choose 2} l_{r}$$ we obtain
$$l_{2} + \frac{3}{4} l_{3} \geq n + \frac{1}{8} \bigg({ n \choose 2} - l_{2} - 3 l_{3} - 6 l_{4}  \bigg).$$
Simple manipulations give
$$\frac{9}{8}(l_{2} + l_{3} + l_{4}) \geq \frac{9}{8} l_{2} + \frac{9}{8} l_{3} + \frac{6}{8} l_{4} \geq \frac{n(n+15)}{16},$$
so we finally obtain 
$$l_{2} + l_{3} + l_{4} \geq \frac{n(n+15)}{18}.$$
\end{proof}
\subsection{Simplicial line arrangements}
Let $\mathcal{A} = \{H_{1}, ..., H_{d}\} \subset \mathbb{R}^{n}$ be a central arrangement of $d$ hyperplanes. We say that $\mathcal{A}$ is simplicial if every connected component of $\mathbb{R}^{n} \setminus \bigcup_{i=1}^{d} H_{i}$ is an open simplicial cone. We say that an arrangement $\mathcal{A} \subset \mathbb{R}^{n}$ is irreducible if $\mathcal{A}$ cannot be expressed as a product arrangement $\mathcal{A}_{1} \times \mathcal{A}_{2}$ with $\mathcal{A}_{1} \subset \mathbb{R}^{\ell}$, $\mathcal{A}_{2} \subset \mathbb{R}^{m}$, and $\ell+m=n$. Using a natural projectivization we can think about rank $n=3$ simplicial hyperplane arrangements as line arrangements in $\mathbb{P}^{2}_{\mathbb{R}}$. Let us recall two properties of simplicial line arrangements.
\begin{enumerate}
    \item From Melchior's proof \cite{Melchior} we see that for any line arrangement $\mathcal{L} \subset \mathbb{P}^{2}_{\mathbb{R}}$ one has $$3-t_{2} + \sum_{r\geq 3}(r-3)t_{r} + \sum_{k\geq 3}(k-3)p_{k}=0,$$
    where $p_{k}$ denotes the number of regions in the complement $\mathbb{P}^{2}_{\mathbb{R}} \setminus \bigcup_{\ell \in \mathcal{L}}\ell$ having $k$ sides. If $\mathcal{L}$ is simplicial, then $p_{k}=0$ for $k\geq 4$, and we have the following equality
    $$t_{2} = 3 + \sum_{r\geq 3}(r-3)t_{r}.$$

\item There is a folklore result providing a bound on the multiplicities of intersection points of irreducible simplicial line arrangements, namely $t_{r} = 0$ if only $r > d/2$ -- see \cite[Proposition 2.1]{Geis} for a modern proof of that result. Observe that the irreducibility assumption is crucial -- a near pencil arrangement $\mathcal{A}$ of $d$ lines is a reducible simplicial arrangement with $t_{d-1} = 1$. The above observation allows us to use freely Langer's inequalities and Bojanowski's inequality (\ref{langorb}) in the irreducible case.
\end{enumerate}
Now we present some very recent and interesting results from the PhD thesis of Geis \cite{Geis}. We start with an observation which gives a bound on multiplicities of singular points of a certain class of simplicial line arrangements \cite[Remark 2.13 iv]{Geis}.
\begin{proposition}
Let $\mathcal{L}$ be a simplicial line arrangement in $\mathbb{P}^{2}_{\mathbb{R}}$ such that $t_{2} \geq t_{3}$ and $t_{i} = 0$ for $i \not\in \{2,3, x\}$. Then $x \leq 8$.
\end{proposition}
\begin{proof}
Since $t_{2} \geq t_{3}$ and by Bojanowski's inequality (\ref{Bojan}), one has the following chain of inequalities:
$$\frac{7}{4}t_{2} \geq t_{2} + \frac{3}{4}t_{3} \geq d + \frac{x(x-4)}{4}t_{x} = d + \frac{x(x-4)(t_{2}-3)}{4(x-3)},$$
where the last equality follows from Melchior's inequality for simplicial line arrangements. Assume now that $x\geq 9$, which implies that
$$0 \leq \bigg( \frac{x(x-4)}{4(x-3)} - \frac{7}{4}\bigg)t_{2} \leq \frac{3x(x-4)}{4(x-3)} - d.$$
This allows us to deduce that
$$d \leq \frac{3x(x-4)}{4(x-3)} \leq \frac{3x}{4} \leq \frac{3d}{8} < d,$$
a contradiction.
\end{proof}
Next, we present an application of one of Langer's inequalities providing a quadratic lower bound on ${\rm max}(t_{2},t_{3})$ for simplicial arrangements \cite[Theorem 5.2]{Geis}.
\begin{theorem}
Let $\mathcal{L}$ be a simplicial line arrangement in $\mathbb{P}^{2}_{\mathbb{R}}$. Then
$${\rm max}(t_{2},t_{3}) > \bigg\lceil \frac{d^{2} + 3d}{27} \bigg\rceil.$$
\end{theorem}
In order to provide you some intuition behing this result, let us recall that Erd\H{o}s and Purdy \cite{ErPu} proved that if $\mathcal{L}$ is an arrangement of $d\geq 25$ lines in the real projective plane such that $t_{d}=0$, then
$${\rm max}(t_{2},t_{3}) \geq d-1.$$
Moreover, they also proved that if $t_{2} < d-1$, then $t_{3}\geq cd^{2}$ for some positive constant $c$.

Before we finish this section, it is worth presenting a Melchior-type inequality for simplicial line arrangements also showed by Geis \cite[Lemma 5.2 c]{Geis} -- the key advantage of this result is that it provides constraints on the number of triple points.
\begin{proposition}
Let $\mathcal{L}$ be an irreducible simplicial line arrangement in $\mathbb{P}^{2}_{\mathbb{R}}$. Then
$$t_{3} \geq 4 + \sum_{r\geq 5}(r-4)t_{r}.$$
\end{proposition}
It is worth mentioning that the proof provided by Geis does not engage any Hirzebruch-type inequalities.

Concluding this section, if we combine Melchior's inequality with Bojanowski's inequality for irreducible simplicial line arrangements, we obtain the following chain of inequalities
$$t_{3} + \frac{4}{3}t_{4}+t_{5} \geq \frac{4}{3}(d-3) + \frac{1}{3}\sum_{r \geq 6}\bigg(r^{2}-8r+12\bigg)t_{r} \geq \frac{4}{3}(d-3),$$
which seems to be an interesting observation. Of course this inequality is sharp.

\section{Generalizations of Hirzebruch's inequalities for plane curve arrangements}
In this section, we present some natural generalizations of Hirzebruch's inequality for line arrangements in the context of higher degree plane curves. We start with the following definition.
\begin{definition}
Let $\mathcal{C} = \{C_{1}, ..., C_{k}\}$ be an arrangement of irreducible curves in the complex projective plane. We say that $\mathcal{C}$ is a $d$-arrangement if the following conditions hold:
\begin{enumerate}
\item all irreducible components $C_{i}$ are smooth and of the same degree $d\geq 1$,
\item all intersection points are ordinary singularities (i.e., these look locally like intersections of lines),
\item there is no point where all curves meet simultaneously. 
\end{enumerate}
\end{definition}
As we can observed, $d$-arrangements are higher degree generalizations of line arrangements, for instance $2$-arrangements will be called conic arrangements, even if in general conic arrangements might have non-ordinary intersection points, for instance tacnodes. The first result presents a Hirzebruch-type inequality for $d$-arrangements \cite[Theorem 2.3]{Pokora1}.
\begin{theorem} Let $\mathcal{C} \subset \mathbb{P}^{2}_{\mathbb{C}}$ be a $d$-arrangement of $k\geq 4$ curves with $d\geq 2$. Then
$$(5d^{2} - 6d)k + t_{2} + \frac{3}{4}t_{3} \geq \sum_{r\geq 5}(r-4)t_{r}.$$
\end{theorem}
It is natural to ask whether one can find an improvement of the above inequality, for instance in order to have the so-called quadratic right-hand side. It turns out that this can be achieved with help of Langer's ideas around his version of the orbifold Miyaoka-Yau inequality \cite[Theorem 2.2]{Pokora2}.
\begin{theorem}
\label{degree:d}
Let $\mathcal{C} = \{C_{1}, ..., C_{k}\} \subset \mathbb{P}^{2}_{\mathbb{C}}$ be a $d$-arrangement of $k\geq 3$ curves with $d \geq 2$. Then
$$t_{2} + \frac{3}{4}t_{3} + d^{2}k(dk-k-1) \geq \sum_{r\geq 5}\bigg(\frac{r^{2}}{4} - r\bigg)t_{r}.$$
\end{theorem}
It is also an interesting question whether one can extend Hirzebruch-type inequalities in the context of arrangements admitting different degrees of irreducible curves. Probably the first result in this spirit is devoted to conic-line arrangements in the complex projective plane having only ordinary singularities \cite[Theroem 2.1]{Pokora2}.

\begin{theorem}
Let $\mathcal{LC} = \{\ell_{1}, ..., \ell_{l}, C_{1}, ..., C_{k}\}$ be an arrangement of $l$ lines and $k$ conics such that $t_{r} = 0$ for $r > \frac{2(l+2k)}{3}$, and we assume that all intersection points of the arrangement are ordinary singularities. Then
$$t_{2} + \frac{3}{4}t_{3} + (4k+2l-4)k \geq l + \sum_{r\geq 5}\bigg(\frac{r^{2}}{4}-r\bigg)t_{r}.$$
\end{theorem}

Finally we consider an interesting topological generalization of line arrangements in the real projective plane --  pseudolines arrangements. 
\begin{definition}
An arrangement $\mathcal{C} \subset \mathbb{P}^{2}_{\mathbb{R}}$ of $d\geq 3$ smooth closed curves is an arrangement of {\it pseudolines} if:
\begin{itemize}
\item all intersection points are transversal, i.e.,  locally can be described as  $x_{1}x_{2} = 0$,
\item every pair of pseudolines intersect at exactly one point,
\item there is no point where all curves meet.
\end{itemize}
\end{definition}
For such arrangements, topological in nature, Shnurnikov proved the following inequality \cite{Shnurnikov}.
\begin{theorem}[Shnurnikov]
Let $\mathcal{C}$ be a pseudoline arrangement of $d\geq 5$ curves such that $t_{d} = t_{d-1} = t_{d-2} = t_{d-3} = 0$. Then
$$t_{2} + \frac{3}{2}t_{3} \geq 8 + \sum_{r\geq 4}(2r-7.5)t_{r}.$$
\end{theorem}
There exists exactly one combinatorial type of pseudoline arrangements for which we obtain equality in Shnurnikov's inequality, namely $d=7$ with $t_{4}=2$ and $t_{2} = 9$.

Let us emphasize that pseudoline arrangements can be viewed algebraically as rank $3$ simple oriented matroids \cite{Bokowski}.

\section{Speculations}
In this short section, we consider possible combinatorial approaches towards Hirzebruch-type inequalities. It is a notoriously difficult question whether we can show any Hirzebruch-type inequality using only elementary combinatorial methods \cite{Research,EP}. At this moment, unfortunately, it seems to be out of reach. However, we can translate this problem using different languages. One of the most promising is the language of tropical geometry, we refer to \cite{Gathmann} for a short introduction to the subject, or to the recent textbook \cite{Sturm}. Let $\mathcal{C}$ be an arrangement of smooth curves in the complex projective plane, and let $\overline{\mathcal{C}}$ denote its tropicalization. Of course it might happen that our curves are intersecting along segments (even not bounded segments), but instead of that we can use the notion of stable intersections in order to avoid such situations. This idea leads to a tropical model of curve arrangements in the complex projective plane. Now we would like to formulate some problems.
\begin{problem}
Is it possible to show a Hirzebruch-type inequality using the language of tropical geometry, or its tropical variation?
\end{problem}
\begin{problem}
Is it possible to find tropical analogues of the Bogomolov-Miyaoka-Yau inequality?
\end{problem}
\section*{Acknowledgement} The idea to write up this survey was born after email correspondence between the author and Frank de Zeeuw, and this was a very stimulating exchange. I would like to thank David R. Wood for suggestions and remarks that allowed me to improve this survey, to Terence Tao for pointing out \cite{Solymosi}, to Renzo Cavalieri for his remarks about the tropical world, to S\l awomir Rams for his remarks, and to Igor Dolgachev for reading the manuscript.  I would like to thank David Geis for allowing me to reproduce some of the results from his PhD thesis, and for useful remarks. The content of this survey was presented during a small workshop in Loughborough in January 2019, and I would like to thank Hamid Ahmadinezhad and Misha Rudnev for their invitation and for giving me a great possibility to discuss many open problems with the participants -- a special word of thanks goes to Adam Sheffer and Konrad Swanepoel. Finally, I strongly thank anonymous referees for their tremendous work that allowed me to improve this survey.

\bigskip
   Address:  
   Department of Mathematics,
   Pedagogical University of Cracow,
   Podchor\c a\.zych 2,,
   PL-30-084 Cracow, Poland. \\
\nopagebreak
\textit{E-mail address:} \texttt{piotrpkr@gmail.com, piotr.pokora@up.krakow.pl}
\end{document}